\documentclass[12pt,a4paper]{article}
\usepackage[12pt]{extsizes}
\usepackage[T2A]{fontenc}
\usepackage[cp1251]{inputenc}
\usepackage[english]{babel}
\usepackage{amssymb,amsfonts,amsmath,mathtext}
\usepackage{cite,enumerate,float,indentfirst}
\usepackage{tikz}
\usepackage{calc}
\usepackage{textcase}
\usepackage{amsthm}
\usepackage{tocloft}
\usepackage[left=3cm,right=2cm,top=2cm,bottom=2cm,bindingoffset=0cm]{geometry}

\newtheorem{thm}{Theorem}
\newtheorem{thm*}{Theorem}

\newtheorem{lemm}{Lemma}

\DeclareMathOperator{\diag}{diag}
\makeatletter
\renewcommand{\@biblabel}[1]{#1.}
\makeatother
\newcommand{\Arg}{\mathop{\mathrm{Arg}}\nolimits}
\newcommand{\Log}{\mathop{\mathrm{Log}}\nolimits}

\title{Analytic continuation for solutions to the system \\
	of trinomial algebraic equations}
\author{Irina Antipova,\, Ekaterina  Kleshkova\, and Vladimir Kulikov}
\date{}

\begin{document}
	\maketitle
	\begin{abstract}
In the paper, we deal with the problem of getting analytic continuations for the monomial function associated with a solution to the reduced trinomial algebraic system. 
In particular, we develop the idea of applying the Mellin-Barnes integral representation of the monomial function for solving the extension problem and demonstrate how to achieve the same result following the fact that the solution to the universal trinomial system is polyhomogeneous. As a main result, we construct Puiseux expansions (centred at the origin) representing analytic continuations of the monomial function.
	\end{abstract}

\bigskip


	\section{Introduction}
	
	  We consider a system of  $n$ trinomial algebraic equations of the form
	\begin{equation}
	\label{tri}
	\sum\limits_{\alpha\in A^{(i)}} a_\alpha^{(i)} y^\alpha = 0,\,\, i=1,\ldots , n,
	\end{equation}
	with unknown $y=(y_1,\dots, y_n)\in \left({\mathbb C}\setminus 0\right)^n$ and variable coefficients $a_{\alpha}^{(i)}$, where $A^{(i)}\subset{\mathbb Z}^n$ are fixed three-element subsets and $y^{\alpha}=y_{1}^{\alpha_1}\cdot \ldots \cdot y_{n}^{\alpha_n}$ is a monomial. Without loss of generality we assume, that all sets  $A^{(i)}$ contain the zero element $\bar{0}$ (this may be achieved by dividing the $i$th equation in (\ref{tri}) by a monomial with the exponent in $A^{(i)}$, see the system (\ref{ptiv_isk_sis}) below). We call (\ref{tri}) the {\it universal trinomial system} since any trinomial algebraic system is a result of the substitution of polynomials in new variables for coefficients $a_{\alpha}^{(i)}$.
	
When $ n = 1 $, the system (\ref{tri}) is a scalar trinomial equation. It has a special status in the centuries-old history of algebraic equations. As early as 1786, Bring proved that every quintic polynomial could be reduced to the trinomial form $ y ^ 5 + ay + b $ using the Tschirnhaus transformation. At the turn of the XIX-XX centuries, the dependence of norms of roots on coefficients of the trinomial equation with fixed support was actively studied. Although algebraic characterisation of the mentioned dependence was given by  Bohl in 1908 yet, the geometric view on the problem has been formed much later.  In the recent study by Theobald and de Wolff \cite{ThdeW}, a geometrical and topological characterisation for the space of univariate trinomials was provided by reinterpreting the problem in terms of amoeba theory.

	Of particular interest is the {\it reduced system} of $n$ trinomial equations
	\begin{equation}
	\label{ptiv_isk_sis}
	y^{\omega^{(i)}}+x_{i} y^{\sigma^{(i)}}-1=0,\,\, i=1,\ldots, n,
	\end{equation}
	with unknown~$y=(y_1, \ldots, y_n)$, supports  of equations $A^{(i)}:= \{ \omega^{(i)}, \sigma^{(i)}, \overline{0} \} \subset \mathbb{Z}^n_{\geqslant}$ and variable complex coefficients~$x=(x_1, \ldots, x_n)$. It is assumed that a matrix~$\omega$, composed of column vectors $\omega^{(1)}, \ldots, \omega^{(n)}$, is nondegenerate.

    Let $y(x)=(y_1(x),\ldots ,y_n(x))$ be a multivalued algebraic vector-function of solutions to the system (\ref{ptiv_isk_sis}). We call a branch of $y(x)$ defined by conditions $y_i(0)=1,\,\, i=1, \ldots, n$ the {\it principal solution} to the system (\ref{ptiv_isk_sis}).  Determined the principal solution $y(x)$ we consider the following monomial function
	\begin{equation}\label{mon}
	{y}^d(x):= y_1^{d_1}(x) \cdot \ldots \cdot y_n^{d_n}(x), \,\, d=(d_1, \ldots, d_n) \in \mathbb{Z}^n_{+}.
	\end{equation}
	Our goal is to obtain Puiseux expansions (centred at the origin) representing analytic continuations of the Taylor series for the monomial ~$y^d(x)$ of the principal solution to the system ~(\ref{ptiv_isk_sis}). Puiseux type parametrisations of an algebraic variety via the amoeba of the discriminant locus of the variety canonical projection were studied in \cite{ArSa}. The existence of such parametrisations for plane curves was proved by Puiseux \cite{Pu}: this fact is known as the Newton--Puiseux theorem which states that one can find local parametrisations of the form $x=t^k$, $y=\varphi(t)$, where $\varphi$ is a convergent power series. We aim at investigating  Puiseux expansions for analytic continuations of (\ref{mon}) which may fail to "recognize" some pieces of the discriminant set. It means that the convergence domain $G$ of a series projects onto the domain $\Log (G)$ containing a certain collection of connected components of the discriminant amoeba complement. An example in Section 2 illustrates how the series converging in the preimage ${\Log}^{-1}(E_0)$ of the component $E_0$ of the amoeba complement admits an analytic continuation to the domain $G$ for which $\Log (G)$ covers components $E_1$, $E_2$ and an amoeba tentacle separating them, see Figure \ref{ris:pic_am}. This analytic continuation is given by another series expansion.
	
	When $n=1$,  analytic continuations for the Taylor series of the principal solution to the universal algebraic equation (not necessarily a trinomial) were found in \cite {AMi}, where the Mellin--Barnes integral representation for the solution was used as a tool of the analytic continuation. This integral, with indicating the convergence region of it,  was wholly studied in \cite {A07}. While a power series converges in a polycircular domain, a Mellin-Barnes integral converges in a sectorial domain which is defined only by conditions for arguments ${\arg}\, x_i$ of variables $x_i$. Remark that the intersection of these domains is always nonempty. Consequently, a series expansion of the solution to the equation admits an analytic continuation into the sectorial domain by means of the integral. Of course, we may follow this approach to getting  analytic extensions for the monomial (\ref{mon}) in a case when the corresponding Mellin-Barnes integral represents it. Therein we can obtain analytic continuations of the Taylor series in the form of Puiseux series via the multidimensional residues technique. 
		
However, we can get the same series following the fact that the solution $y(a)$ to the system (\ref{tri}) is polyhomogeneous. For the trinomial system (\ref{tri}) it means that via some monomial transformation of coefficients the system can be reduced to  (\ref{ptiv_isk_sis}) or to another reduced system which, similarly,  has the only one variable coefficient in each equation. We perceive any reduced system of equations as the general (homogeneous) system (\ref {tri}) written in suitable coordinates. The transition from a reduced system to another one enables us to obtain series continuations for monomials of coordinates of solutions to these systems.
    
   The paper is organised as follows. In Section 2, we review the technique of the calculation of multidimensional Mellin-Barnes integrals which is based on the separating cycle principle formulated in \cite{Ts92} (see also \cite{ZhTs}). We present an example which illustrates what computational issues can arise in this way of getting analytic extensions. In Section 3, we discuss the procedure of the dehomogenization (reduction) of the system (\ref{tri}) and yield the Taylor series expansions for the monomials of the principal solutions to all reduced systems associated with the system (\ref{tri}).  Theorem 1 gives these series as a result of the application of the logarithmic residue formula \cite{Yuzhakov} and the linearization procedure for each reduced system. The idea of using the logarithmic residue formula for getting the Taylor expansions was realised in \cite{KuSt}, where the special instance of the reduced polynomial system, with the diagonal matrix $\omega$, was considered.  In Section 4, we use Taylor expansions derived in Theorem 1 and appropriate monomial transformations to obtain the desired Puiseux series which are supposed to be the analytic continuations of the Taylor series for the monomial  $y^d(x)$ of the principal solution to (\ref{ptiv_isk_sis}) (Theorem 2). Finally, we discuss the example from Section 2 again in terms of the result of Theorem 2.

	\section{Mellin-Barnes integral as a tool of analytic continuation}
	
Traditionally, the Mellin-Barnes integral is regarded as the inverse Mellin transform for special meromorphic functions, which are rations of products of a finite number of superpositions of gamma functions with affine functions. Their role in the theory of algebraic equations was revealed first by Mellin in~\cite{Me}, where he wrote without any proof the integral representation for the solution to the universal algebraic equation later investigated in  \cite{A07}.  In our study we consider such integrals in the extended sense, having in mind the presence of the polynomial factor in the integrand besides gamma-functions. 

The Mellin integral transform for monomials of a solution to the reduced polynomial system was studied in \cite{A03} and \cite{St03}.
Following \cite{St03}, we match the Mellin-Barnes integral 
	\begin{equation}
	\label{int_M_B}
	\frac{1}{ {\left( 2 \pi i \right)}^n } \int \limits_{ \gamma + i \mathbb{R}^n } \prod \limits_{j=1}^{n} \frac{ \Gamma { \left( z_j \right) } \Gamma { \left( \frac{d_j}{\omega_j} - \frac{1}{\omega_j} { \langle \sigma_j, z \rangle } \right) } }{ \Gamma { \left( \frac{d_j}{\omega_j} - \frac{1}{\omega_j} { \langle \sigma_j, z \rangle } + z_j +1 \right) } } Q(z) x^{-z}\, dz
	\end{equation}
	to the monomial $y^d(x)$. In (\ref{int_M_B}) $x^{-z}$ denotes the product $x_1^{-z_1} \cdot \ldots \cdot x_n^{-z_n}$,  $\sigma_j$ is the  $j$th row of the matrix $\sigma$ composed of column vectors $\sigma^{(1)},\dots , \sigma^{(n)}$,  $\gamma$ belongs to the domain
	\begin{equation*}
	U= \{ u \in \mathbb{R}_{+}^n: { \langle \sigma_j,u \rangle } < d_j, \,\, j=1, \ldots, n \},
	\end{equation*}
	and $Q(z)$ is a polynomial represented by the determinant
	\begin{equation*}
	\label{mn_Q}
	Q(z)= \frac{1}{ \det \omega } \det {\left| \left| \delta^j_i { \left( d_j - { \langle \sigma_j, z \rangle } \right) } + \sigma_j^{(i)} z_i \right| \right| }_{i,j=1 }^{n}, 
	\end{equation*}
	 where  $\delta_i^j$ is the Kronecker symbol. Here is assumed that $\omega$ is a diagonal matrix with elements $\omega_1, \ldots , \omega_n$ on the diagonal.
	
	 Remark that the integral ~({\ref{int_M_B}}) can have the empty convergence domain.  Following \cite{Ky17}, this integral has the nonempty convergence domain if and only if all the diagonal minors of the matrix $\sigma$ are positive. In this case, the integral~(\ref{int_M_B}) represents the monomial $y^d(x)$ of the principal solution to the system~(\ref{ptiv_isk_sis}), and it can be used as a tool of the constructive analytic continuation of power series. 
	
	Let us show how to calculate the integral ({\ref{int_M_B}}). A method of the calculation is based on the separating cycle principle formulated in \cite{Ts92} and developed in \cite{ZhTs}. This principle deals with the calculation of the Grothendieck-type integrals
	\begin{equation}\label{i_gr}
	\frac{1}{(2\pi i)^n}\int\limits_{\Delta_g}\frac{h(z)dz}{f_1(z)\ldots f_n(z)},
	\end{equation}
	where the integration set $\Delta_g$ is the skeleton of the polyhedron $\Pi_g$ associated with the holomorphic proper mapping $g:(g_1,\ldots , g_n):{\mathbb C}^n\to {\mathbb C}^n$, and the integrand has poles on divisors $D_j=\{z: f_j(z)=0\}$, $j=1,\ldots , n$. One means by the polyhedron $\Pi_g $ the preimage $ g ^ {- 1} (G) $ of the domain $ G = G_1 \times \ldots \times G_n $, where each $ G_j $ is a domain  on the complex plane with the piecewise smooth boundary. We associate a facet	$\sigma_j=\{z:g_j(z)\in\partial G_j,\,\, g_k(z)\in G_k,\,\, k\ne j\}$ of the polyhedron $\Pi_g$ with each  $j\in\{1,\ldots , n\}$.

	{\bf Definition.} {\it A polyhedron $\Pi_g$ is said to be compartible with the set of divisors $\{D_j\}$, if for each $j=1,\ldots , n$ the corresponding facet  $\sigma_j$ of the polyhedron $\Pi_g$ does not intersect the divisor $D_j$ .}
	
	Assume further that the intersection $Z=D_1\cap \ldots \cap D_n$ is discrete. The local residue with respect to the family of divisors $\{D_j\}$ at each point $a\in Z$ (the Grothendieck residue) is defined by the integral
	(see \cite{Ts92})
	$$
	{\mbox {\rm res}_{f,a}}\Omega=\frac{1}{(2\pi i)^n}\int\limits_{\Gamma_{a}(f)}^{}\Omega,
	$$
	where $\Omega$ is the integrand in (\ref{i_gr}), and $\Gamma_{a}(f)$ is a cycle given in the neiborhood $U_a$ of the point $a$ as follows
	$$
	\Gamma_{a}(f)=\{z\in U_a: |f_1(z)|=\varepsilon_1,\ldots , |f_n(z)|=\varepsilon_n\},\,\, \varepsilon_j<< 1.
	$$
	If $a$ is a simple zero of the mapping $f$, i.e. the Jacobian $J_f=\partial f/\partial z$ is nonzero at the point $a$, then the local residue is calculated by the formula
	\begin{equation}\label{f_Ca}
	{\mbox {\rm res}_{f,a}}\Omega=\frac{h(a)}{J_f(a)}.
	\end{equation}
	
{\bf Theorem (principle of separating cycles).} {\it If the polyhedron $\Pi_g$ is bounded and compatible with the family of polar divisors $\{D_j\}$, then the integral (\ref{i_gr}) equals to the sum of Grothendieck residues in the domain $\Pi_g$.}	
	
	One can reduce the integral (\ref{int_M_B}) to the canonical form (\ref{i_gr}) in the following way. We interpret the vertical integration subspace 	$\gamma + i \mathbb{R}^n$ as the skeleton of some polyhedron. For instance, in the case $n=1$, it can be the skeleton of only two polyhedra: the right and left halfplanes with the separating line $\gamma + i \mathbb{R}$.  For $n>1$ this subspace may serve as the skeleton of an infinite number of polyhedra. Our objective is to divide all the set of $2n$ families of polar hyperplanes of the integral  (\ref{int_M_B})
\begin{equation*}
	\begin{array}{l}
	L_{j}: z_j=-\nu,\,\, \\
	L_{n+j}: \frac{d_j}{\omega_j} - \frac{1}{\omega_j} { \langle \sigma_j, z \rangle }=-\nu, \,\, j=1,\ldots , n,\,\, \nu\in {\mathbb Z_{\geqslant}}
	\end{array}
\end{equation*}	
 into $n$ divisors and construct a polyhedron compatible with this family of divisors. We consider polyhedra of the type
 $$
 \Pi_g=\{z\in {\mathbb C}^n: {\mbox{\rm {Re}}}g_j(z)<r_j,\,\, j=1,\ldots, n\},
 $$
where $g_j(z)$ are liner functions with real coefficients. It is clear, that $\Pi_g=\pi +i{\mathbb R}^n$, where $\pi$ is a simplicial $n$--dimensional cone in the real subspace ${\mathbb R}^n\subset {\mathbb C}^n$. Remark that in the case of an unbounded polyhedron, besides the compatibility condition of the polyhedron and polar divisors, one should require a sufficiently rapid decrease of the integrand $\Omega$ in the polyhedron $\Pi_g$. For the integral (\ref{int_M_B}) the non-confluence property provides the decrease of the integrand, see \cite{PSTs05} and \cite{ZhTs}. We recall that the non-confluence property for the hypergeometric Mellin-Barnes integral means that sums of coefficients of the variable $z_j$ over all gamma-factors in the numerator and the denominator are equal.

Now, applying the technique discussed above, we construct analytic continuations for the solution to the following system of equations
\begin{equation}	
\label{eq1}
\left\{
\begin{aligned}
y_1^4+x_1y_1^2y_2-1=0,\\
y_2^4+x_2y_1y_2^2-1=0.\\
\end{aligned}
\right.
\end{equation}
	
For the description of the convergence domains of power series and Mellin-Barnes integrals we introduce the following mappings from $({\mathbb C}\setminus 0)^n$ into ${\mathbb{R}}^n$: 
	\begin{equation*}
			\Log : (x_1, \ldots, x_n) \longrightarrow \left(\log {|x_1|}, \ldots , \log {|x_n|} \right),
	\end{equation*}
	\begin{equation*}
			\Arg : (x_1, \ldots, x_n) \longrightarrow \left(\arg x_1, \ldots , \arg x_n \right).
	\end{equation*}

	The monomial $y_1(x)\cdot y_2(x)$ of the principle solution to the system (\ref{eq1}) admits the Taylor series representation
	\begin{equation}
	\label{RT_ex}
	\sum \limits_{\vert k \vert \geqslant 0} \frac{{(-1)}^{\vert k \vert}}{k!} \frac{ \Gamma { \left( \frac{1}{4}+  \frac{1}{2}k_1 + \frac{1}{4} k_2 \right) }  \Gamma { \left( \frac{1}{4} + \frac{1}{4} k_1 + \frac{1}{2} k_2 \right) } }{ \Gamma { \left( \frac{5}{4} - \frac{1}{2} k_1 + \frac{1}{4} k_2 \right) }  \Gamma { \left( \frac{5}{4} + \frac{1}{4} k_1 - \frac{1}{2} k_2 \right) } } \frac{1}{16} {(1+k_1+k_2)} x_1^{k_1} x_2^{k_2},
	\end{equation}
	which converges in some neighborhood of the origin, see Theorem 1 below. In turn, the Mellin-Barnes integral of the form
	\begin{equation}
	\label{eq:25}
	\frac{1}{{(2\pi i)}^2} \int \limits_{\gamma+i \mathbb{R}^2} \frac{ \Gamma(z_1) \Gamma(z_2) \Gamma \left( \frac{1}{4}- \frac{1}{2}z_1 -\frac{1}{4}z_2\right) \Gamma \left( \frac{1}{4}- \frac{1}{4}z_1- \frac{1}{2}z_2 \right) }{ \Gamma \left( \frac{5}{4}+ \frac{1}{2}z_1-\frac{1}{4}z_2 \right) \Gamma \left( \frac{5}{4}- \frac{1}{4}z_1+\frac{1}{2}z_2 \right)} \frac{\left(1-z_1-z_2\right)}{16}x^{-z} dz,
	\end{equation}
	where $\gamma$ is a point in the open quadrangle
	\begin{equation*}
	U= \left\{  {u\in \mathbb{R}^2_+ : 2u_1+u_2<1,\,\,u_1+2u_2<1 } \right\},
	\end{equation*}	
	represents the monomial $y_1(x)\cdot y_2(x)$ in a sectorial domain ${\Arg}^{-1}(\Theta)$ determined by
	\begin{equation}
	\label{eq:27}
	\Theta= \left\{ (\theta_1, \theta_2)\in \mathbb{R}^2:\,\, |\theta_1| < \frac{\pi}{2}, \,\, |\theta_2| < \frac{\pi}{2}, |2\theta_2 - \theta_1| < \frac{3\pi}{4}, |\theta_2 - 2\theta_1| < \frac{3\pi}{4} \right\},
	\end{equation}
	here $\theta_1=\arg{x_1}, \theta_2=\arg{x_2}$.  Figure 4 shows the domain $\Theta$ which is the interior of the convex octagon. The general description of convergence domains of multiple Mellin-Barnes integrals gives Theorem 4.4.25 in the book \cite{SdTs}. Thus, the integral (\ref{eq:25}) gives the analytic continuation of the series (\ref{RT_ex}) into the sectorial domain ${\Arg}^{-1}(\Theta)$ .

    We next calculate the integral (\ref{eq:25}) using the principle of separating cycles. It admits a representation as a sum of local residues of the integrand
	\begin{equation}
	\label{forma}
	\Omega =  \frac{ \Gamma(z_1) \Gamma(z_2) \Gamma \left( \frac{1}{4}- \frac{1}{2}z_1 -\frac{1}{4}z_2\right) \Gamma \left( \frac{1}{4}- \frac{1}{4}z_1- \frac{1}{2}z_2 \right) }{ \Gamma \left( \frac{5}{4}+ \frac{1}{2}z_1-\frac{1}{4}z_2 \right) \Gamma \left( \frac{5}{4}- \frac{1}{4}z_1+\frac{1}{2}z_2 \right) } \frac{\left(1-z_1-z_2\right)}{16}x_1^{-z_1}x_2^{-z_2} dz_1 \, dz_2
	\end{equation}
	in some polyhedron, which contains the vertical imagine integration subspace $\gamma + i \mathbb{R}^2$ as the skeleton. Furthermore, the polyhedron and polar divisors of  $\Omega$ should satisfy the compatibility conditions. 
	
	The form $\Omega$ has four families of polar complex lines:
	\begin{equation}\label{pp}
	\begin{split}
	&L_1: \, z_1=-\nu , \\
	&L_2: \, z_2=- \nu ,{}\\
	&L_3: \, \frac{1}{4}- \frac{1}{4}(2z_1+z_2)=- \nu ,{}\\
	&L_4: \, \frac{1}{4}- \frac{1}{4}(z_1+2z_2)=- \nu , \,\,\, \nu \in \mathbb{Z}_{ \ge }.
	\end{split}
	\end{equation}
	
	Figures \ref{ris:picU} and \ref{ris:picU_2} show the intersection of the real subspace with families (\ref{pp}), and also with
	\begin{equation*}
		\begin{split}
	&L_5: \, \frac{5}{4}+ \frac{1}{2}z_1-  \frac{1}{4} z_2=- \nu ,{}\\
	&L_6: \, \frac{5}{4}- \frac{1}{4}z_1+ \frac{1}{2} z_2=- \nu ,
		\end{split}
	\end{equation*}
	which are polar sets of gamma-functions in the denominator of the form (\ref{forma}). The quadrangle $U$, the point $\gamma$ belongs to, is coloured in grey.

	\begin{figure}[H]
			\begin{center}
			\begin{minipage}[h]{0.4\linewidth}
			\begin{center}	
			\includegraphics[width=1.1\textwidth]{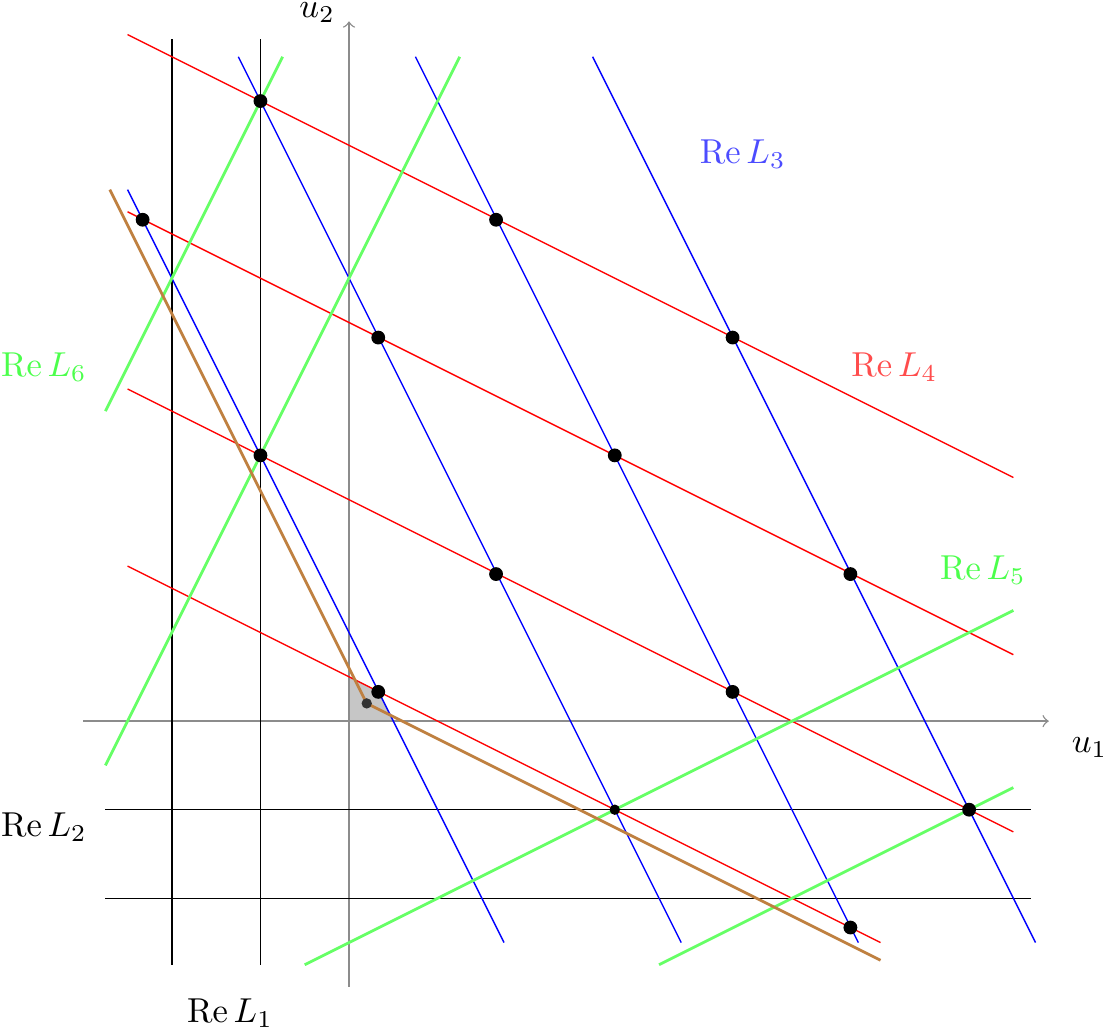}
			\caption{\small{The real section of polar divisors. The cone $\pi_1$}.}
			\label{ris:picU}
			\end{center}
			\end{minipage}
			\hfill
				\begin{minipage}[h]{0.5\linewidth}
				\begin{center}
				\includegraphics[width=0.98\textwidth]{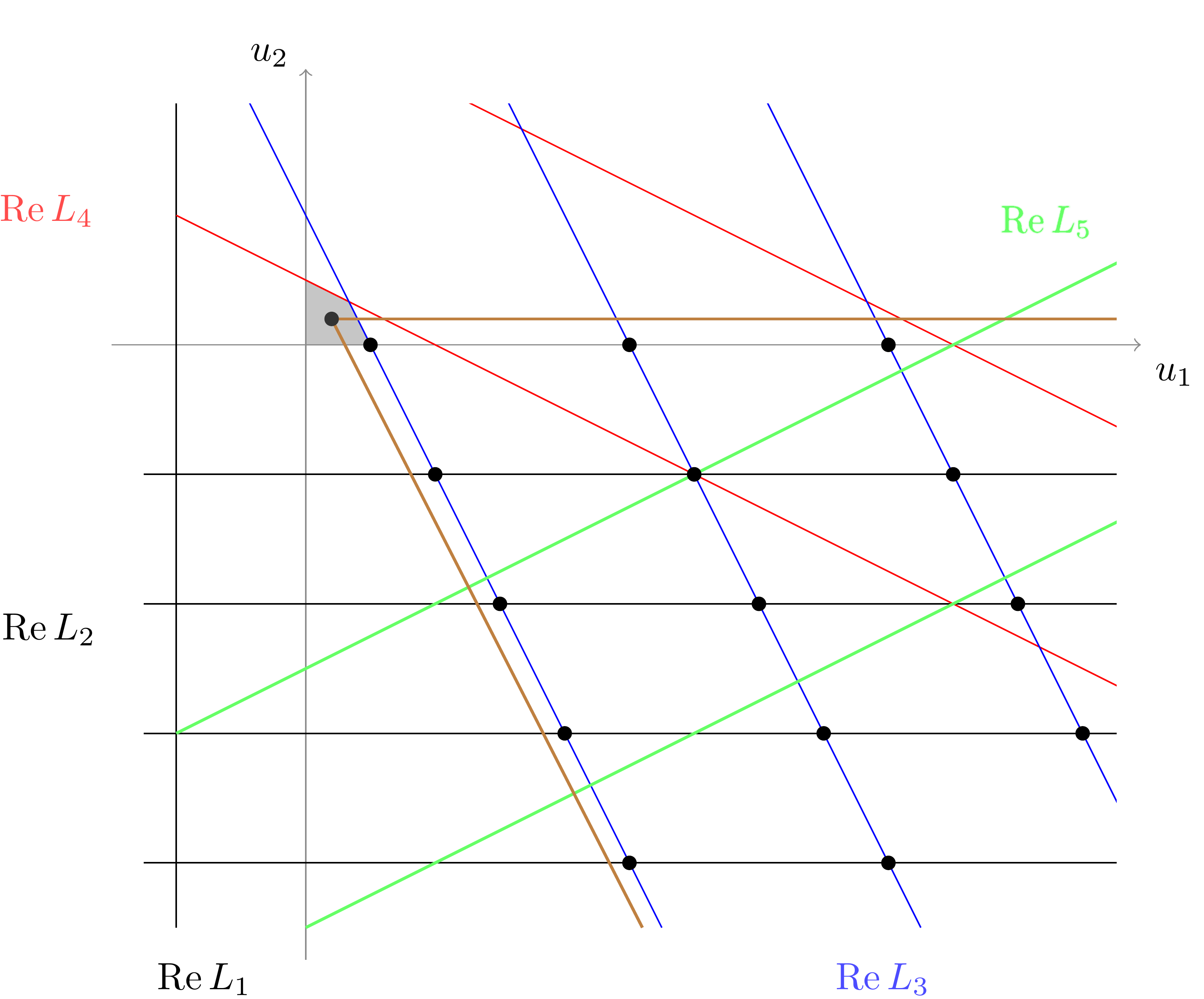}
			\caption{\small{The real section of polar divisors. The cone $\pi_2$}.}
			\label{ris:picU_2}
			\end{center}
			\end{minipage}
		\end{center}					
		\end{figure}

	At first, given the set of all polar lines of the integrand $\Omega$, we form two divisors  $D_1=\{ L_2, L_3 \}$ and  $D_2= \{ L_1, L_4 \}$. We next construct a polyhedron  $\Pi_1 = \pi_1 + i \mathbb{R}^2 $ compartible with this set of divisors, with the skeleton $\gamma + i \mathbb{R}^2$. Figure \ref{ris:picU} shows a two-dimensional cone (sector) $\pi_1\subset \mathbb{R}^2$ generated by rays which are parallel to the real sections of $L_3$ and $L_4$. It forms the polyhedron $\Pi_1$. Secondly, we consider divisors $D_1^{'}= \{ L_3, L_4 \}$ and $D_2^{'}=\{ L_2 \}$. A cone $\pi_2$  generated by rays which are parallel to the real sections of $L_3$ and $L_2$ forms a polyhedron $\Pi_2 = \pi_2 + i \mathbb{R}^2$, compartible with the set of divisors $D_1^{'}$, $D_2^{'}$, see Figure \ref{ris:picU_2}.
	
	We can see in Figure \ref{ris:picU} that families  $ L_5, L_6 $ as well as $ L_1, L_2, L_3, L_4 $ come into the polyhedron $\Pi_1$ , so in the cone $ \pi_1 $ there are points at which two, three and even four lines intersect. However, the form  $\Omega$ has nonzero residues only at points $z(k)=\left( z_1(k), z_2(k) \right)$ with coordinates
	\begin{equation}\label{poi_1}
	 	\begin{split}
	 	&z_1(k)=\frac{1}{3}+ \frac{8}{3}k_1-  \frac{4}{3} k_2,{}\\
	 	&z_2(k)=\frac{1}{3}- \frac{4}{3}k_1+  \frac{8}{3} k_2, \,\, k=(k_1, k_2) \in \mathbb{Z}^{2}_{\geqslant}.
	 	\end{split}
	 \end{equation}
	 
\noindent The intersection points (\ref{poi_1}) of lines $L_3$, $L_4$ are indicated in Figure \ref{ris:picU} by a black colour. Hence, the sum of local residues at points $z(k)$ yields the Puiseux series
	 \begin{equation}
	\label{raz}
	P_1(x)=\frac{1}{x_1^{1/3} x_2^{1/3} } \sum \limits_{k\in \mathbb{Z}^2_{\geqslant}} c_k x_1^{-8/3k_1+4/3k_2} x_2^{4/3k_1-8/3k_2}
	\end{equation}
	with coefficients
	\begin{equation}
	\label{c_k_p1}
	c_k= \frac{{(-1)}^{|k|}}{k!} \frac{ \Gamma{\left( \frac{1}{3}+\frac{8}{3} k_1 -\frac{4}{3} k_2 \right)} \Gamma{\left( \frac{1}{3}-\frac{4}{3} k_1 +\frac{8}{3}k_2 \right)} }{ \Gamma{\left( \frac{4}{3} +\frac{5}{3}k_1 -\frac{4}{3} k_2 \right)} \Gamma{\left( \frac{4}{3}-\frac{4}{3} k_1 +\frac{5}{3}k_2 \right)} } \frac{1}{9} {( 1-4k_1-4k_2 )}.
	\end{equation}
	
    Four families of lines $L_2, L_3, L_4$ and $L_5$ come into the polyhedron  $\Pi_2$, see Figure 2.  However, the form  $\Omega$ has nonzero residues only at points $z(k)=\left( z_1(k), z_2(k) \right)$ with coordinates
	\begin{equation}\label{poi_2}
	 	\begin{split}
	 	&z_1(k)=\frac{1}{2}+ 2k_1+  \frac{1}{2} k_2,{}\\
	 	&z_2(k)=-k_2, \,\, k=(k_1, k_2) \in \mathbb{Z}_{\geqslant}.
	 	\end{split}
	 \end{equation}
	 
\noindent Points (\ref{poi_2}) are black in Figure 2, where lines $L_2$, $L_3$ intersect. The sum of residues at $z(k)$ yields the Puiseux series
	 \begin{equation}
	\label{raz_2}
	P_2(x)=\frac{1}{x_1^{1/2}} \sum \limits_{k\in \mathbb{Z}^2_{\geqslant}} c_k x_1^{-2k_1-1/2k_2} x_2^{k_2}
	\end{equation}
	with coefficients
	\begin{equation}
	\label{c_k_p2}
	c_k= \frac{{(-1)}^{|k|}}{k!} \frac{ \Gamma{\left( \frac{1}{2}+2 k_1 +\frac{1}{2} k_2 \right)} \Gamma{\left( \frac{1}{8}-\frac{1}{2} k_1 +\frac{3}{8}k_2 \right)} }{ \Gamma{\left( \frac{3}{2} +k_1 +\frac{1}{2} k_2 \right)} \Gamma{\left( \frac{7}{8}-\frac{1}{2} k_1 -\frac{5}{8}k_2 \right)} } \frac{1}{16} {( 1-4k_1+k_2 )}.
	\end{equation}
	
	We remark that arguments of $\Gamma$--functions in coefficients of the series (\ref{RT_ex}) and also in  (\ref{c_k_p1}) and (\ref{c_k_p2}) can be real nonpositive numbers, which are poles for the function $\Gamma$. So by a ration of two $\Gamma$--functions we mean a meromorphic function with removable singularities at that points. For instance, we mean
	\begin{equation*}
	\frac{ \Gamma(-1) }{ \Gamma(0) }=\frac{ \Gamma(-1) }{ -\Gamma(-1) }=-1.
	\end{equation*}
	So, series (\ref{raz}) and (\ref{raz_2}) are analytic extensions of the series (\ref{RT_ex}).
	
We now characterise domains of convergence of Puiseux series obtained above in the logarithmic scale. According to the two-sided Abel lemma for hypergeometric series \cite{PSTs05}, there exist the relationship between the structure of the convergence domain of this series and its support. Since  series (\ref{raz}) and (\ref{raz_2}) represent branches of the multivalued algebraic function $y_1(x)\cdot y_2(x)$ with singularities on the discriminant set of the system (\ref{eq1}), projections of convergence domains of such series on the space of variables $\log |x_1|, \log |x_2|$ are unions of several components of the discriminant amoeba complement  for the system, see Figure \ref{ris:pic_am}. We recall that the amoeba of the algebraic set $V \subset {\mathbb C}^n$ is defined to be the image of $V$ under the mapping $\Log$. In this way, series (\ref{raz}) converges in the domain $G_1= \Log^{-1} (E_3)$, where $E_3$ is an amoeba complement component. The projection ${\Log} (G_2)$ of the convergence domain  $G_2$ of the series (\ref{raz_2}) covers two components $E_1$, $E_2$ and an amoeba tentacle separating them, see Figure \ref{ris:pic_am}.

	\begin{figure}[H]
			\begin{center}
				\begin{minipage}[h]{0.45\linewidth}
				\begin{center}
				\includegraphics[width=1\textwidth]{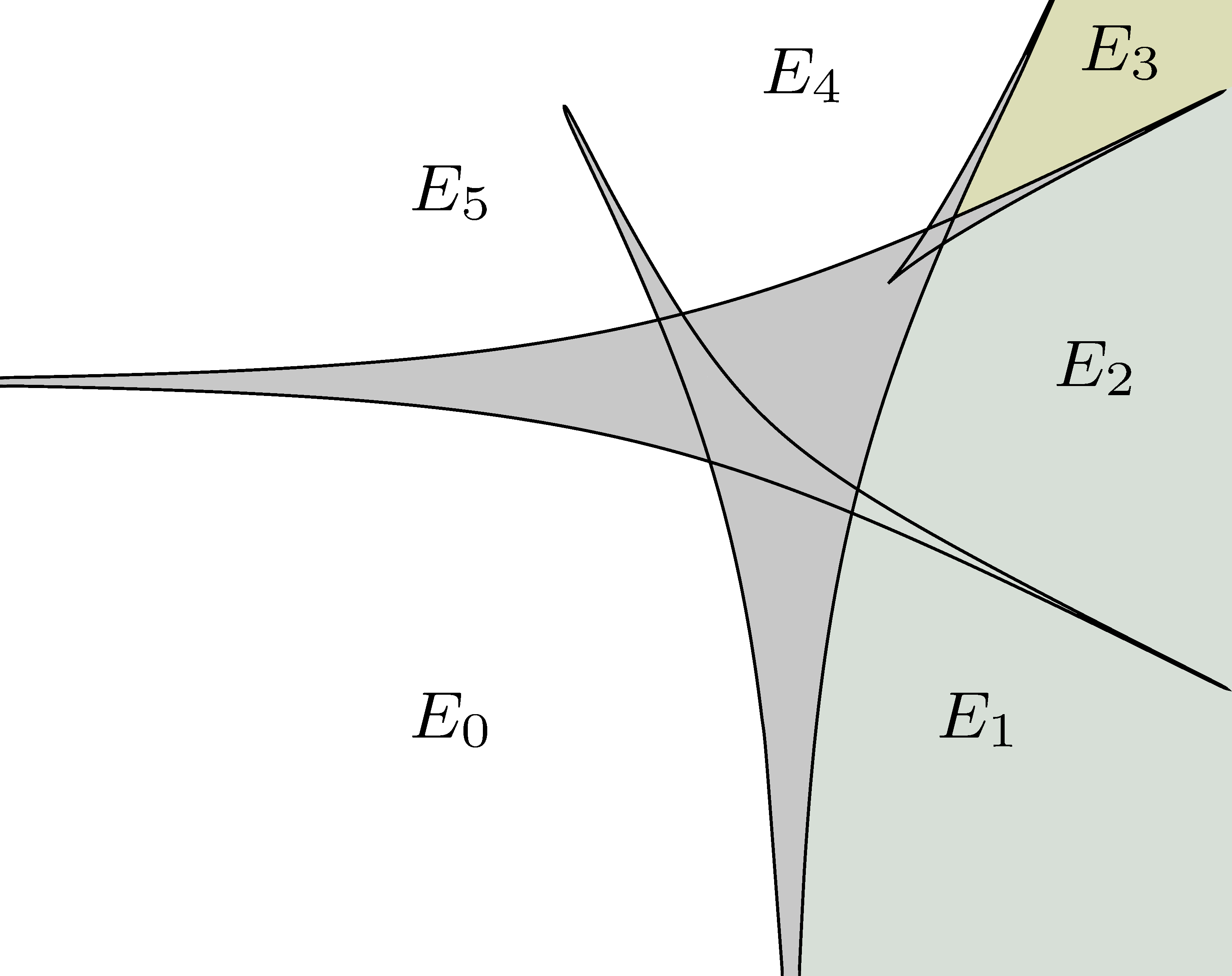}
				\caption{\small{The discriminant amoeba of the system (\ref{eq1}) and its complement components $E_\nu$.}}
				\label{ris:pic_am}
				\end{center}
				\end{minipage}
				\hfill
				\begin{minipage}[h]{0.45\linewidth}
					\begin{center}
					\includegraphics[width=0.8\textwidth]{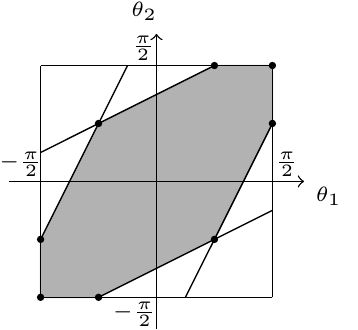}
					\caption{\small{The domain $\Theta.$ \newline \qquad}}
					\label{ris:picO_S}
					\end{center}
				\end{minipage}
			\end{center}					
	\end{figure}

	\section{Taylor series for monomials of solutions to reduced systems}
	
	We consider the system of $n$ trinomials (\ref{tri}) with unknowns~$y_1, \ldots, y_n$, variable coefficients $a=(\ldots, a^{(i)}_\alpha, \ldots)$ and the set of supports $A^{(1)},\ldots, A^{(n)},$ the same as the system~\eqref{ptiv_isk_sis} has. 
	
Let us denote by $A$ the disjunctive union of sets~$A^{(i)}$.  It consists of $3n$ elements, and  we interpret it as the $(n \times 3n)$ -- matrix
\begin{equation*}
A={\left( A^{(1)}, \ldots, A^{(n)} \right)}= { \left( \alpha^1, \ldots, \alpha^{3n} \right) },
\end{equation*}
with columns~$\alpha^k$ which are exponents of monomials of the system~(\ref{tri}). We order elements~$\alpha \in A$, and, correspondingly, coefficients~$a^{(i)}_{\alpha}$, $\alpha\in A$ of the system~(\ref{tri}). The set of coefficients  $a=(a_{\alpha})$ is a vector space~$\mathbb{C}^A \simeq \mathbb{C}^{3n}$.
	
	The system~\eqref{tri} can be reduced by an appropriate change of coefficients in such a way that only one variable coefficient remains in each equation, and the other ones will be constant as in the system~\eqref{ptiv_isk_sis}.  Therein,  supports ~$A^{(1)}, \ldots, A^{(n)}$ remain the same, and the solution to the system~(\ref{tri}) can be recovered by the solution to any reduced system.  On the whole, the reduction procedure (dehomogenization) of the system is based on the polyhomogeniety property of the solution $y(a)=(y_1(a), \ldots, y_n(a))$, which can be expressed as follows:
        \begin{equation}
        \label{polyhomogeneous}
	y{ \left( \ldots \lambda_0^{(i)} \lambda^{\alpha} a_{\alpha}^{(i)} \ldots \right) }={ \left( \lambda_1^{-1} y_1 {\left( .. a_{\alpha}^{(i)} .. \right) }, \ldots, \lambda_n^{-1} y_n {\left( .. a_{\alpha}^{(i)} .. \right) } \right) },
	\end{equation}
	where $\lambda_0 = \left( \lambda_0^{(1)}, \ldots, \lambda_0^{(n)} \right)$, $\lambda=(\lambda_1, \ldots, \lambda_n)\in { \left( \mathbb{C} \setminus 0 \right) }^n$, see \cite{ATs12}.

	In each set~$A^{(i)}$ we fix a pair of elements~$\mu^{(i)},\,\, \nu^{(i)}$ and form the $n \times n$--matrix
	\begin{equation}
	\label{matrica}
	\varkappa := { \left( \mu^{(1)}- \nu^{(1)}, \ldots, \mu^{(n)}- \nu^{(n)}  \right) }
	\end{equation}
	with columns~$\mu^{(i)}- \nu^{(i)}$. The matrix~$\varkappa$ is assumed to be nondegenerate. Each fixed set of $n$ pairs~$\mu^{(i)},\, \nu^{(i)}$ corresponds to the reduced system of trinomials 
	\begin{equation}
	\label{pri_sis}
	r_{\beta^{(i)}}^{(i)}y^{\beta^{(i)}} + y^{\mu^{(i)}} - y^{\nu^{(i)}}=0,\,\,\, i=1,\ldots, n,
	\end{equation}
	with new unknown~$y=(y_1, \ldots, y_n)$, $\beta^{(i)}\in A^{(i)}$ and variable coefficients~$r={ \left( r_{\beta^{(i)}}^{(i)} \right) } \in  \mathbb{C}^n$. In each set~$A^{(i)}$, we can choose an unordered pair~$\mu^{(i)}, \nu^{(i)}$ in three ways. Hence, we consider at most $3^n$ ways of the reduction the system~(\ref{tri}) to the form~\eqref{pri_sis}.
	If $\mu^{(i)}=\omega^{(i)}$, $\nu^{(i)}=\overline{0}$ and $\beta^{(i)}=\sigma^{(i)}$ for all $i\in\left\{1, \ldots, n\right\}$, then we get the system (\ref{ptiv_isk_sis}).
	
	Consider a branch of the solution to the system~(\ref{pri_sis}) under condition $y_i(0)=1$ and call it the principle solution. For the vector $d=(d_1, \ldots, d_n) \in \mathbb{R}^n_+$ we introduce the monomial function ${y}^d(r):=y_1^{d_1}(r) \cdot \ldots \cdot y_n^{d_n}(r)$ of coordinates of the principal solution to the system~(\ref{pri_sis}). Concerning the system (\ref{pri_sis}), we use the following notations:  ~$\beta$ is the matrix formed of columns $\beta^{(i)}$ and  $\overline{\beta}$ is the matrix with columns $\left( \beta^{(i)} - \nu^{(i)} \right)$. Moreover, the symbol $\Gamma(b)$  we will use for the short writing of the product $\prod\limits_{k=1}^{n} \Gamma\left(b_k\right)$, where $b=\left(b_1,\ldots,b_n\right)$ is a vector. The diagonal matrix with components of the vector $b$  on the main diagonal we denote by ${\mbox{\text{diag}}}[b]$ and the $I$ denotes the vector with unit coordinates.

	\begin{thm}
	\label{thm_1}
	The monomial~${y}^d(r)$ of the principle solution to the system~(\ref{pri_sis}) admits the Taylor series representation
	with coefficients
	\begin{equation}
	\label{kf_ck_2}
	c_k=\frac{{(-1)}^{ \vert k \vert }}{k!} \frac{ \Gamma { \left( \varkappa^{-1} d + \varkappa^{-1} \overline{\beta} k \right) } }{ \Gamma { \left( \varkappa^{-1} d + \varkappa^{-1} \overline{\beta} k -k + I \right) } } Q(k),\,\, k\in{\mathbb Z}_{\geqslant}^{n},
	\end{equation}
	where $Q(k)$ is the determinant of the matrix ${\left( \diag {  \left[ \varkappa^{-1} d + \varkappa^{-1} \overline{\beta} k \right] }  - { \varkappa^{-1}  \overline{\beta} \, \diag {[k]}} \right)}$,  $k!:=k_1! \cdot \ldots \cdot k_n!$ and $\mid k \mid:=k_1 + \ldots +k_n$.
	\end{thm}

	\begin{proof}
Following \cite{ATs12}, we carry out the linearization of the system \eqref{pri_sis}. For that we regard  \eqref{pri_sis} as a system of equations in the space $\mathbb{C}^n_r\times\mathbb{C}^n_y$ with coordinates $r=(r_{\alpha^{(i)}}^{(i)}),$ $y=\left( y_1,\ldots,y_n\right)$, and introduce in $\mathbb{C}^n\times\mathbb{C}^n$ the change of variables  $(\xi,W) \rightarrow (r,y)$ by setting
\begin{equation}
y = W^{- \varkappa^{-1}},
\;\;
r= \xi \odot W^{\varkappa^{-1} \bar\beta - E},
\label{linearize}
\end{equation}
where $\xi=(\xi_1,\dots , \xi_n)$, $W=(W_1,\dots , W_n)$, $\odot$ denotes the Hadamard (coordinate-wise) product and $E$ is the unit matrix. As a result of this change of variables, the system \eqref{pri_sis} can be written in the vector form as follows   
\begin{equation}
W=\xi+I.
\label{lin_system}
\end{equation}
Equations of the system \eqref{lin_system} are linear, so the change of variables \eqref{linearize} is called the linearization. Coordinates of the solution to the system \eqref{pri_sis} in new variables $\xi=(\xi_1, \ldots, \xi_n)$, $W=(W_1, \ldots, W_n)$ take the form
$$ y_j(r(\xi)) = \left(W_1, \ldots, W_n\right)^{-\left(\varkappa^{-1}\right)^{(j)}} ,$$
where $W_i=1+\xi_i,$  $\left(\varkappa^{-1}\right)^{(j)}$ is the $j$th column of the inverse matrix $\varkappa^{-1}$ for the matrix $\varkappa$.

	We represent the inversion $\xi(r)$ of the linearization \eqref{linearize} as an implicit mapping given by the following set of equations
	\begin{equation}
	\label{map_F}
	F(\xi,r) = 
	\left(F_1(\xi,r),\ldots,F_n(\xi,r) \right) = 
	\xi \odot W^{\varkappa^{-1} \bar\beta - E} - r = 0.
	\end{equation}
Calculate the vector $y(\xi)$ at the value of the mapping $\xi(r)$. To this end, following the idea implemented in \cite {KuSt} for a system of polynomials with a diagonal matrix $ \omega $, we apply the logarithmic residue formula, see \cite[Th. 20.1, 20.2]{Yuzhakov}. It yields the following integral

	\begin{equation*}
	\label{po}
{y}^d(r) = \frac{1}{(2\pi i)^n}\int\limits_{\Gamma_\varepsilon}\frac{{y}^d(\xi) \Delta(\xi) d\xi}{F(\xi,r)},
	\end{equation*}
where $\Gamma_\varepsilon = \{\xi\in \mathbb{C}^n : |\xi_j|=\varepsilon, j=1,\ldots,n\},$ $\Delta(\xi)$ is the Jacobian of the mapping (\ref{map_F}) with respect to $\xi$ and $F(\xi, r)$ denotes the product $F_1(\xi,r) \cdot \ldots \cdot F_n (\xi, r)$. The radius $\varepsilon$ we choose in such a way that the corresponding polycylinder lies outside the zero set of the Jacobian~$\Delta(\xi)$.     

	\begin{lemm}
	The Jacobian of the mapping $F(\xi,r)$ with respect to $\xi$ is

	\begin{equation*}
	\Delta(\xi) = 
	W^{(\varkappa^{-1} \bar\beta) I - 2I}
	\det \left( E+\diag[\xi]\varkappa^{-1}\bar\beta \right).
	\end{equation*}
	\end{lemm}

	\begin{proof}[Proof of Lemma 1]

The $j$th component of the mapping $F(\xi,r)$ has the following form:
$$
F_j = F_j(\xi,r)= \xi_j \prod\limits_{k=1}^{n} W_k^{(\varkappa^{-1} \bar\beta - E)_k^{(j)}} - r_j.
$$
The calculation of the derivative of $F_j$ with respect to $\xi_j$ looks as follows:
\begin{equation*}
\begin{split}
\frac{\partial F_j}{\partial \xi_j} =& \prod\limits_{k=1}^n W_k^{(\varkappa^{-1}\bar\beta - E)_k^{(j)}}+\xi_j (\varkappa^{-1}\bar\beta - E)_j^{(j)} \prod\limits_{k=1}^n W_k^{(\varkappa^{-1}\bar\beta - E)_k^{(j)} - \delta_k^j}=
\\
=&(1+\xi_j (\varkappa^{-1} \bar\beta)^{(j)}_j) \prod\limits_{k=1}^n W_k^{(\varkappa^{-1} \bar\beta)_k^{(j)}-2\delta_k^j},
\end{split}
\end{equation*}
and the derivative with respect to $\xi_i$, when $i\neq j$, is equal to

\begin{equation*}
\frac{\partial F_j}{\partial \xi_i} =  \xi_j (\varkappa^{-1}\bar\beta - E)_i^{(j)} \prod\limits_{k=1}^n W_k^{(\varkappa^{-1}\bar\beta)_k^{(j)} - \delta_k^j - \delta_k^i},
\end{equation*}
where $\delta_k^j$, $\delta_k^i$ denote the Kronecker symbols.

Extracting common factors in the rows and columns of the obtained determinant, we get the assertion of the lemma.
	\end{proof}

Remark that at the origin the Jacoby matrix for the mapping $F(\xi, r)$ is the unit matrix. Hence, the Jacobian $\Delta(\xi)$ does not vanish in the neighbourhood of the origin and conditions of Theorems 20.1, 20.2 from \cite{Yuzhakov} hold.

The monomial $y^d(r)$ after the change of variables takes the following form:

	\begin{equation*}
y^d(\xi) = W^{-\varkappa^{-1}d}.
	\end{equation*}

\noindent Consequently, application of the logarithmic residue formula yields the integral representation:
	\begin{equation}\label{i_logres}
y^d (r) = \frac{1}{{(2 \pi i)}^n} \int\limits_{\Gamma_\varepsilon} \frac{ W^{-\varkappa^{-1}d + (\varkappa^{-1} \bar\beta) I - 2 I}}{F(\xi,r)}\det \left( E+\diag[\xi]\varkappa^{-1}\bar\beta \right) d\xi.
	\end{equation}

Expand the kernel of the integral (\ref{i_logres}) into a multiple geometric series. To this end, we use the coordinate notations:

\begin{equation*}
\begin{split}
y^d(r) = \frac{1}{{(2 \pi i)}^n} \int\limits_{\Gamma_\varepsilon} \frac{ W^{-\varkappa^{-1}d + (\varkappa^{-1} \bar\beta) I - 2 I}}{
\prod\limits_{j=1}^{n}\left(\xi_j \prod\limits_{k=1}^{n} W_k^{(\varkappa^{-1} \bar\beta - E)_k^{(j)}} - r_j\right)
}\det \left( E+\diag[\xi]\varkappa^{-1}\bar\beta \right) d\xi \\
=\frac{1}{{(2 \pi i)}^n} \int\limits_{\Gamma_\varepsilon} 
\frac{
W^{-\varkappa^{-1}d + (\varkappa^{-1} \bar\beta) I - 2 I}
}
{
W^{(\varkappa^{-1}\bar\beta) I - I}  \prod\limits_{j=1}^{n} \xi_j \left(1 - \frac{r_j}{\xi_j \prod\limits_{k=1}^{n} W_k^{(\varkappa^{-1} \bar\beta - E)_k^{(j)}}}\right)
}
\det \left( E+\diag[\xi]\varkappa^{-1}\bar\beta \right) d\xi \\
=\frac{1}{{(2 \pi i)}^n} \int\limits_{\Gamma_\varepsilon} 
\frac{
W^{-\varkappa^{-1}d - I}
}
{
\prod\limits_{j=1}^{n} \xi_j \left(1 - \frac{r_j}{\xi_j \prod\limits_{k=1}^{n} W_k^{(\varkappa^{-1} \bar\beta - E)_k^{(j)}}}\right)
}
\det \left( E+\diag[\xi]\varkappa^{-1}\bar\beta \right) d\xi.
\end{split}
\end{equation*}

	Since there exists such a number $\delta$ that for all $\xi \in \Gamma_\varepsilon$ and $\Vert r \Vert < \delta$ the inequality
	$$
	\frac{r_j}{\xi_j \prod\limits_{k=1}^{n} W_k^{(\varkappa^{-1} \bar\beta - E)_k^{(j)}}}<1
	$$
is valid, the integral (\ref{i_logres}) admits the following representation: 

	\begin{equation*}
y^d(r)=\frac{1}{{(2 \pi i)}^n} \int\limits_{\Gamma_\varepsilon} 
\frac{
W^{-\varkappa^{-1}d  - I} \det \left( E+\diag[\xi]\varkappa^{-1}\bar\beta \right) 
}
{
\prod\limits_{j=1}^{n} \xi_j 
}
\left(\sum\limits_{k\in\mathbb{Z}_\geqslant^n }\prod\limits_{j=1}^n\left(\frac{r_j}{\xi_j W^{(\varkappa^{-1} \bar\beta - E)^{(j)}}}\right)^{k_j}\right)
d\xi.
	\end{equation*}

Changing the order of summation and integration in the last integral, we get the series

	\begin{equation*}
y^d(r)=
\sum\limits_{k\in\mathbb{Z}^n_{\geqslant} }
\left(
\frac{1}{{(2 \pi i)}^n} \int\limits_{\Gamma_\varepsilon} 
\frac{
W^{-\varkappa^{-1}(d+ \bar\beta k)   + k- I}
}
{
\xi^{k+I}
}
\det \left( E+\diag[\xi]\varkappa^{-1}\bar\beta \right)
d\xi \right)r^k.
	\end{equation*}

\noindent The coefficient $ c_k $ of the series is determined by the expression in parentheses. It can be calculated by the Cauchy integral formula. As a result, we get:

	\begin{equation*}
c_k =
\frac{1}{k!}
\frac{\partial^k}{\partial \xi^k}
\left(
W^{-\varkappa^{-1}(d+ \bar\beta k)   + k- I}
\det \left( E+\diag[\xi]\varkappa^{-1}\bar\beta\right) \right)\Big\vert_{\xi=0}.
	\end{equation*}
We bring the factor  $W^{-\varkappa^{-1}(d+ \bar\beta k)   + k- I}$ into the determinant in such a way that each row of it still to depend on one variable $\xi_j$. We obtain

	\begin{equation*}
c_k =
\frac{1}{k!}
\frac{\partial^k}{\partial \xi^k}
\det 
\left(
\diag\left[W^{\diag\left[-\varkappa^{-1}(d+ \bar\beta k)   + k- I\right]}\right]\times
\left(E+\diag[\xi]\varkappa^{-1}\bar\beta\right)
\right) 
\Big\vert_{\xi=0}.
	\end{equation*}
We next use the multilinearity property of the determinant and the fact that each row depends only on one variable $\xi_j$. As a result, we have

	\begin{equation*}
c_k =
\frac{1}{k!}
\det 
\left\Vert
\left.
\frac{\partial^{k_j}}{\partial \xi_j^{k_j}}
W_j^{\left(-\varkappa^{-1}(d+ \bar\beta k)\right)_j+k_j-1}\left(\delta_i^j + \xi_j (\varkappa^{-1}\bar\beta)_j^{(i)}\right)\right\vert_{{\xi_j}=0}
\right\Vert_{i,j=1}^n
.
	\end{equation*}

Finally, we perform calculations in the above determinant:
\begin{equation*}
\begin{split}
&\left.
 \frac{\partial^{k_j}}{\partial \xi_j^{k_j}}
W_j^{\left(-\varkappa^{-1}(d+ \bar\beta k)\right)_j+k_j-1}\left(\delta_i^j + \xi_j (\varkappa^{-1}\bar\beta)_j^{(i)}\right)\right\vert_{{\xi_j}=0}
\\
=& (-1)^{k_j} 
\left(\left(\varkappa^{-1}(d+ \bar\beta k)\right)_j\delta_i^j - k_j (\varkappa^{-1}\bar\beta)_j^{(i)}\right)  
\prod\limits_{m=1}^{k_j-1} \left( \left(\varkappa^{-1}(d+ \bar\beta k)\right)_j-k_j+m \right)
\\
=&
(-1)^{k_j}\frac{ \Gamma\left( \left(\varkappa^{-1}(d+ \bar\beta k)\right)_j \right)}{\Gamma\left( \left(\varkappa^{-1}(d+ \bar\beta k)\right)_j-k_j+1 \right) }\left(\left(\varkappa^{-1}(d+ \bar\beta k)\right)_j\delta_i^j - k_j(\varkappa^{-1}\bar\beta)_j^{(i)}\right).
\end{split}
\end{equation*}

\noindent Taking out the common factor in each row of the determinant and taking into account the factor ~$\frac{1}{k!}$, we get the view of the coefficient $c_k$ declared in formula~(\ref{kf_ck_2}).
	\end{proof}

	Coefficients of the Taylor series for the monomial~${y}^d (x)$ of the principal solution to the system~(\ref{ptiv_isk_sis}) one can find by formula~(\ref{kf_ck_2}) setting $\varkappa = \omega$, $\overline{\beta} = \sigma$. Thus, the series is as follows:
	\begin{equation}
	\label{rad_it}
	{y}^d (x)=\sum \limits_{k \in \mathbb{Z}^n_{\geqslant}} \frac{{(-1)}^{ \mid k \mid  }}{k!} \frac{ \Gamma {( \omega^{-1} d + \omega^{-1} \sigma k )} }{ \Gamma {( \omega^{-1} d + \omega^{-1} \sigma k - k +I )} } P(k)x^k,
	\end{equation}
	where $P(k)=\det {\left( \diag {  \left[ \omega^{-1} d + \omega^{-1} \sigma k \right] }  - { \omega^{-1}  \sigma \, \diag {[k]}} \right)}$.
	\section{Puiseux series }
	We fix $n$ couples $\mu^{(i)}, \,\, \nu^{(i)} \in A^{(i)}$ of exponents of the system~(\ref{ptiv_isk_sis}) and compose the matrix
	\begin{equation*}
	\varkappa= { \left( \varkappa^{(i)}_j \right) } = { \left( \mu^{(i)}_j - \nu^{(i)}_j \right) },
	\end{equation*}
	assuming that it is nondegenerate. In accordance with the choice of the set of pairs $\mu^{(i)}$, $\nu^{(i)}$, let us devide the set $\{ 1, \ldots , n \}$ on three disjoint subsets:
	\begin{equation}
	\label{3_pod}
	\begin{split}
	&J= \{ j: \nu^{(j)}=\overline{0},\,\, \mu^{(j)}=\omega^{(j)} \},\\{}
	&L= \{ l: \nu^{(l)}=\overline{0},\,\, \mu^{(l)}=\sigma^{(l)} \},\\{}
	&T= \{ t: \nu^{(t)}=\sigma^{(t)},\,\, \mu^{(t)}=\omega^{(t)} \}.
	\end{split}	
	\end{equation}
	
	We introduce two matrices
	\begin{equation*}
	\Phi:= \varkappa^{-1} \cdot \sigma, \,\,\, \Psi:= \varkappa^{-1} \cdot \omega,
	\end{equation*}
	with rows $\varphi_1, \ldots, \varphi_n$ and $\psi_1, \ldots, \psi_n$ respectively. Moreover, we consider truncated rows
	\begin{equation*}
	\varphi_l^J,\, \psi_l^L,\, \psi_l^T,\,\, l\in L,
	\end{equation*}
	\begin{equation*}
	\varphi_t^J,\, \psi_t^L,\, \psi_t^T,\,\, t\in T,
	\end{equation*}
	which consist of entries of rows $\varphi_l$, $\psi_l$, $l\in L$ and $\varphi_t$, $\psi_t$, $t\in T$ indexed by elements of sets $J,\,L$ and $T$. Respectively, we introduce truncated vectors $k^J$, $k^L$, $k^T$  for the vector $k=( k_1, \ldots, k_n )$. The scalar product of vectors we denote as follows $\langle \cdot, \cdot  \rangle$.
	
	\begin{thm}
	\label{thm_P}
	For any collection of $n$ couples $\mu^{(i)},\, \nu^{(i)} \in A^{(i)}$ with the nondegeneracy condition of the corresponding matrix $\varkappa$ there exist an analytic continuation of the Taylor series for the monomial $y^d(x)$ of the principal solution to the system~(\ref{ptiv_isk_sis}) in the form of the Puiseux series
	\begin{equation*}
	\sum \limits_{k \in \mathbb{Z}^n_{\geqslant}} \tilde{c_k} x^{m(k)},
	\end{equation*}
	which has the support consisting of points $m(k)=\left( m_1(k), \ldots, m_n(k) \right)$ with coordinates
	
	\begin{equation*}
	\begin{aligned}
	&m_j(k)=k_j, \,\, j \in J,\\
	&m_l(k)= - { \langle \varphi_l^J, k^J \rangle }-{ \langle \psi_l^L, k^L \rangle } + { \langle \psi_l^T, k^T \rangle }- { \langle d, \varkappa^{-1}_l \rangle },\,\, l \in L,\\
	&m_t(k)= { \langle \varphi_t^J, k^J \rangle }+{ \langle \psi_t^L, k^L \rangle } - { \langle \psi_t^T, k^T \rangle }+ { \langle d, \varkappa^{-1}_t \rangle }, \,\, t \in T,
	\end{aligned}
	\end{equation*}
	and coefficients $\tilde{c}_k$ expressed in terms of coefficients (\ref{kf_ck_2}) as follows
	\begin{equation*}
	\tilde{c}_k= e^{i\pi\sum\limits_{t\in T}\left(k_t+m_t(k)\right)}c_k.
	\end{equation*}
	\end{thm}
	
	\begin{proof}
	We start the proof with finding the monomial change of variables ~$r=r(a)$ reducing the system~(\ref{tri}) to the form~(\ref{pri_sis}). To this end, we get the Smith normal form $S_q$ for the matrix~$\varkappa$, multiplying it on the left and right by unimodular matrices~$C$ and $F$ as follows:
	\begin{equation}
	\label{norm_matrix}
	C \varkappa F=S_q,
	\end{equation}
	here the $S_q$ is a diagonal matrix with integers ~$q_1, \ldots, q_n$ on the diagonal, and $q_j \mid q_{j+1},\,\, 1\leqslant j \leqslant n-1$, see \cite{P96}. It follows from (\ref{norm_matrix}) that the inverse matrix $\varkappa^{-1}$ admits the representation
	\begin{equation}
	\label{matrix_kappa}
	\varkappa^{-1}=F S_q^{-1}C.
	\end{equation}
	As it was mentioned above, the solution $y(a)$ of the system~(\ref{tri}) is polyhomogeneous.
	We find the polyhomogeneity parameters $\lambda_0^{(i)}$ and $\lambda=(\lambda_1,\ldots, \lambda_n)$ such that
	\begin{equation}	
	\label{poliod}
		\begin{aligned}
		\lambda_0^{(i)} \lambda^{\mu^{(i)}} a_{\mu^{(i)}}^{(i)}=1,\\
		\lambda_0^{(i)} \lambda^{\nu^{(i)}} a_{\nu^{(i)}}^{(i)}=-1,\\
		\end{aligned}
	\end{equation}
	for $i=1, \ldots, n$. For that, we solve the following system of equations:
	\begin{equation}
	\label{sist_yrav}
	\lambda^{\varkappa^{(i)}} = g_i,\,\, i=1, \ldots, n,
	\end{equation}
	where
	\begin{equation*}
	g_i=-\frac{ a_{\nu^{(i)}}^{(i)} }{ a_{\mu^{(i)}}^{(i)} }.
	\end{equation*}
	Using the relation ~(\ref{matrix_kappa}), we can write the solution of the system~(\ref{sist_yrav}) in the matrix form as follows
	\begin{equation*}
	\lambda=g^{\varkappa^{-1}}=g^{FS_q^{-1} C}={ \left( {\left( g^{f^{(1)}} \right)}^{ \frac{1}{q_1} }, \ldots, {\left( g^{f^{(n)}} \right)}^{ \frac{1}{q_n} } \right) }^C,
	\end{equation*}
	where the vector $g$ has coordinates~$g_i$, and $f^{(1)}, \ldots, f^{(n)}$ are columns of the matrix $F$.
	By choosing for each $i$ all $q_i$ values of the radical~${ \left( g^{f^{(i)}} \right) }^{ \frac{1}{q_i} }$, we yield all branchers of the matrix radical $g^{\varkappa^{-1}} $. 
	There are $\mid \det \varkappa \mid=q_1 \cdot \ldots \cdot q_n$ of them.

	For each $i \in \{ 1, \ldots, n \}$ we find the parameter $\lambda_0^{(i)}$, using one of relations (\ref{poliod}). If $\nu^{(i)}=\overline{0}$, then $\lambda_0^{(i)}= - \frac{1}{a_0^{(i)}}$. For $\mu^{(i)}= \omega^{(i)}$ we get 
	\begin{equation*}
	\lambda_0^{(i)}= \frac{1}{a_{\omega^{(i)}}^{(i)}} \cdot { \left( {\left( g^{f^{(1)}} \right)}^{ \frac{1}{q_1} }, \ldots, {\left( g^{f^{(n)}} \right)}^{ \frac{1}{q_n} } \right) }^{-C \omega^{(i)}}.
	\end{equation*}
	If $i \in J$, then the coefficient $r_{\sigma^{(i)}}^{(i)}$ of the system~(\ref{pri_sis}) can be expressed in terms of coefficients $a$ of the system~(\ref{tri}) in two ways:
	\begin{equation}
	\label{mon_zam}
	\begin{split}
	r_{\sigma^{(i)}}^{(i)}&=-\frac{ a_{\sigma^{(i)}}^{(i)} }{ a_{0}^{(i)} } \cdot { \left( g^{f^{(1)}} \right) }^{ \frac{ \langle c_1, \sigma^{(i)} \rangle }{ q_1 } } \cdot \ldots \cdot { \left( g^{f^{(n)}} \right) }^{ \frac{ \langle c_n, \sigma^{(i)} \rangle }{ q_n } }
	\\
	r_{\sigma^{(i)}}^{(i)}&=\frac{ a_{\sigma^{(i)}}^{(i)} }{ a_{\omega^{(i)}}^{(i)} } \cdot { \left( g^{f^{(1)}} \right) }^{ \frac{ \langle c_1, \sigma^{(i)}-\omega^{(i)} \rangle }{ q_1 } } \cdot \ldots \cdot { \left( g^{f^{(n)}} \right) }^{ \frac{ \langle c_n, \sigma^{(i)}-\omega^{(i)} \rangle }{ q_n } }.
     \end{split}
	\end{equation}
	If $i \in L$, then the coefficient $r_{\omega^{(i)}}^{(i)}$ of the system~(\ref{pri_sis}) can be expressed in terms of coefficients $a$ of the system~(\ref{tri}) as follows
	\begin{equation}
	\label{mon_zam_3}
	r_{\omega^{(i)}}^{(i)}=-\frac{ a_{\omega^{(i)}}^{(i)} }{ a_{0}^{(i)} } \cdot { \left( g^{f^{(1)}} \right) }^{ \frac{ \langle c_1, \omega^{(i)} \rangle }{ q_1 } } \cdot \ldots \cdot { \left( g^{f^{(n)}} \right) }^{ \frac{ \langle c_n, \omega^{(i)} \rangle }{ q_n } }.
	\end{equation}
	For $i \in T$ the relation is as follows
	\begin{equation}
	\label{mon_zam_4}
	r_{\overline{0}}^{(i)}=\frac{ a_{\overline{0}}^{(i)} }{ a_{\omega^{(i)}}^{(i)} } \cdot { \left( g^{f^{(1)}} \right) }^{ -\frac{ \langle c_1, \omega^{(i)} \rangle }{ q_1 } } \cdot \ldots \cdot { \left( g^{f^{(n)}} \right) }^{ -\frac{ \langle c_n, \omega^{(i)} \rangle }{ q_n } }.
	\end{equation}
	In formulae (\ref{mon_zam})--(\ref{mon_zam_4}) vectors $c_1, \ldots, c_n$ are rows of the matrix~$C$.
	
	In particular, if for all $i \in \{ 1, \ldots, n \}$ we choose $\mu^{(i)}=\omega^{(i)}$, $\nu^{(i)}=\overline{0}$, then $L = \varnothing$, $T = \varnothing$ and  $\varkappa=\omega$. The matrix $\omega$ is nondegenerate by assumption and the system~(\ref{pri_sis}) coincides with the system~(\ref{ptiv_isk_sis}). In this case, we get the change of variables $x=x(a)$. It can be written in two ways:
	\begin{equation}
	\label{mon_zam_x}
	\begin{split}
	x_i&=-\frac{ a_{\sigma^{(i)}}^{(i)} }{ a_{\overline{0}}^{(i)} } \cdot { \left( h^{v^{(1)}} \right) }^{ \frac{ \langle u_1, \sigma^{(i)} \rangle }{p_1} } \cdot \ldots \cdot { \left( h^{v^{(n)}} \right) }^{ \frac{ \langle u_n, \sigma^{(i)} \rangle }{p_n} },\\
	x_i&=\frac{ a_{\sigma^{(i)}}^{(i)} }{ a_{\omega^{(i)}}^{(i)} } \cdot { \left( h^{v^{(1)}} \right) }^{ \frac{ \langle u_1, \sigma^{(i)} - \omega^{(i)} \rangle }{p_1} } \cdot \ldots \cdot { \left( h^{v^{(n)}} \right) }^{ \frac{ \langle u_n, \sigma^{(i)} - \omega^{(i)} \rangle }{p_n} }.
	\end{split}
	\end{equation}
	In formulae (\ref{mon_zam_x}) the vector $h$ has coordinates  $h_i=-\frac{a_{\bar{0}}^{(i)}}{a_{\omega^{(i)}}^{(i)}}$,	vectors $u_1,\ldots, u_n$ are rows of the unimodular matrix~$U$, in turn, vectors $v^{(1)}, \ldots, v^{(n)}$ are columns of the unimodular matrix~$V$ such that $\omega =U S_p V,$ where $S_p= \diag [ p_1, \ldots, p_n]$, $p_j \mid p_{j+1}, \,\, 1\leqslant j \leqslant n-1$.
	
	Remark that $g_i=h_i$ for $i \in J$. Furthermore, if $i\in L$ then $g_i=-\frac{a_{\bar{0}}^{(i)}}{a_{\sigma^{(i)}}^{(i)}}$, and for $i\in T$ we have $g_i=-\frac{a_{\sigma^{(i)}}^{(i)}}{a_{\omega^{(i)}}^{(i)}}$. Getting these ratios from (\ref{mon_zam_x}), we substitute the expressions for $g_i$ into 
	 (\ref{mon_zam})--(\ref{mon_zam_4}). As a result, we get coordinates of the monomial transformation~$r=r(x)$ for the transition from the system~(\ref{pri_sis}) to the system (\ref{ptiv_isk_sis}):
	\begin{equation}
	\label{m_z}
		\begin{aligned}
			&r_{\sigma^{(j)}}^{(j)}=x_j \prod \limits_{l \in L} x_l^{- \varphi_l^{(j)} } \cdot \prod \limits_{t \in T} { \left( - x_t \right) }^{ \varphi_t^{(j)} }, \,\, j \in J,\\
			&r_{\omega^{(j)}}^{(j)}=\prod \limits_{l \in L} x_l^{ - \psi_l^{(j)} } \cdot \prod \limits_{t \in T} {\left(- x_t \right) }^{ \psi_t^{(j)} }, \,\, j \in L,\\
			&r_{\overline{0}}^{(j)}=- \prod \limits_{l \in L} x_l^{ \psi_l^{(j)} } \cdot \prod \limits_{t \in T} {\left(- x_t\right)}^{ - \psi_t^{(j)} }, \,\, j \in T.			
		\end{aligned}
	\end{equation}

	
	According to the polyhomogeneity property~(\ref{polyhomogeneous}), the division of the $j$th coordinate of the solution to the system~(\ref{tri}) on $\lambda_j \ne 0$ is compensated by the multiplication of the coefficient~$a_{\alpha}^{(i)}$ on $\lambda^{\alpha}$. So taking into account  (\ref{sist_yrav}) we obtain the relationship between monomials ${y}^d (x)$ and ${y}^d (r)$ of the following form:
	\begin{equation}
	\label{svaz}
	{y}^d (x)= \prod \limits_{j=1}^{n} \frac{ g_j^{\left<  d, \, \varkappa^{-1}_j \right>} }{ h_j^{ \left< d, \, \omega^{-1}_j \right> } } {y}^d (r),
	\end{equation}
	where $\varkappa_j^{-1}$, $\omega_j^{-1}$ are $j$th rows  of matrices $\varkappa^{-1}$ and $\omega^{-1}$ correspondingly. Using relations (\ref{mon_zam_x}), and the fact that $g_j=h_j$ for $j\in J$, we write (\ref{svaz}) as follows:

\begin{equation}
\label{n_y}
y^d(x)= \prod \limits_{l \in L} x_l^{- \left< d,\, \varkappa^{-1}_l \right> } \prod \limits_{t \in T} {\left(  e^{i \pi} x_t \right) }^{\left< d, \,\varkappa^{-1}_t \right>} y^{d}(r). 
\end{equation}
	
Hence, making the substitution (\ref{m_z}) in the expansion (\ref{kf_ck_2}) and taking into account the relation (\ref{n_y}), we conclude, that the support $S$ of the required Puiseux series consists of points $m(k)=\left( m_1(k), \ldots, m_n(k) \right)$ with coordinates

\begin{equation*}
\begin{aligned}
&m_j(k)=k_j, \,\, j \in J,\\
&m_l(k)= - { \langle \varphi_l^J, k^J \rangle }-{ \langle \psi_l^L, k^L \rangle } + { \langle \psi_l^T, k^T \rangle }- { \langle d, \varkappa^{-1}_l \rangle },\,\, l \in L,\\
&m_t(k)= { \langle \varphi_t^J, k^J \rangle }+{ \langle \psi_t^L, k^L \rangle } - { \langle \psi_t^T, k^T \rangle }+ { \langle d, \varkappa^{-1}_t \rangle }, \,\, t \in T.
\end{aligned}
\end{equation*}
	The coefficient $\tilde{c}_k$ of the Puiseux series is expressed in terms of the coefficient  (\ref{kf_ck_2}) by the following formula
	\begin{equation*}
		\tilde{c}_k= e^{i\pi\sum\limits_{t\in T}\left(k_t+m_t(k)\right)}c_k.
	\end{equation*}
	
	\end{proof}

	As it mentions in Section 2, by the two-sided Abel lemma for hypergeometric series \cite{PSTs05} the cone of the support $S$ of the series defines the logarithmic image $\Log (G)$ of the convergence domain $G$ of the series. It means that the geometry of the domain $G$ is closely related to the structure of the amoeba ${\mathcal A}$ of the discriminant hypersurface $\nabla$ of the system (\ref{ptiv_isk_sis}). The amoeba $\mathcal{A}$ can be obtained from the amoeba $\mathcal{A}^{'}$ of the discriminant set of the system ~(\ref{pri_sis}) via the affine transform associated with the change of variables $r=r(x)$.  Consequently, the recession cone of the set $\Log (G)$ for the Puiseux series of the monomial ${y}^d (x)$ is the image of the negative orthant $- \mathbb{R}^n_{+}$ under the affine transform.
	
	In conclusion, we return to the example from Section 2 to make the following remark. By Theorem \ref{thm_P} we associate the Puiseux series (\ref{raz})  with couples of exponents:
	$$(2,1), (0,0)\in A^{(1)},\,\, (1,2), (0,0)\in A^{(2)},$$ 
	and, accordingly, the series (\ref{raz_2}) with the set
	$$(2,1), (0,0)\in A^{(1)},\,\,  (0,4), (0,0)\in A^{(2)}.$$ 
	
	{\bf Acknowledgments:} The first author was supported by the Foundation for the Advancement of Theoretical Physics and Mathematics "BASIS"(no. $\sharp$ 18-1-7-60-1). 
	The second author was supported by the Foundation for the Advancement of Theoretical Physics and Mathematics "BASIS"(no. $\sharp$ 18-1-7-60-2). The third author was supported by the grant of the Ministry of Education and Science of the Russian Federation (no. 1.2604.2017/PCh). 
%


	\addcontentsline{toc}{section}{References}
	\begin{center}
	\renewcommand{\refname} 
	\end{center}
\bigskip

\noindent {\bf Authors' addresses}: 

\noindent Irina Antipova, Siberian Federal University, 79 Svobodny pr., 660041 Krasnoyarsk, Russia, iantipova@sfu-kras.ru

\noindent Ekaterina Kleshkova, ekleshkova@gmail.com

\noindent Vladimir Kulikov, Siberian Federal University, 79 Svobodny pr., 660041 Krasnoyarsk, Russia, v.r.kulikov@mail.ru 	

{\centering References}
	\begin{thebibliography}{99} 
			\bibitem{A03}
		Antipova, I.A.:
		An expression for the superposition of general algebraic functions
		in terms of hypergeometric series.
		Siberian Math. J. \textbf{44}:5, 757-764 (2003)
		
		
		\bibitem{A07}
		Antipova, I.A.:
		Inversion of many-dimensional Mellin transforms and solutions of algebraic equations.
		Sb. Math. \textbf{198}:4, 447-463 (2007)
		
		
		\bibitem{AMi}
		Antipova, I.A., Mikhalkin, E.N.:
		Analytic continuations of a general algebraic function by means of Puiseux series.
		Proc. Steklov Inst. Math. \textbf{279}, 3-13 (2012)
		
		
		\bibitem{ATs12}
		Antipova, I.A., Tsikh, A.K.:
		The discriminant locus of a system of n Laurent polynomials in n variables.
		Izv. Math. \textbf{76}:5, 881-906 (2012)
		
		
		\bibitem{Yuzhakov}
		Aizenberg, L.A., Yuzhakov, A.P.:
		Integral representations and residues in multidimensional complex analysis.
		Translations of Mathematical Monographs, vol. 58, p. 283. AMS (1983)
		
		
		\bibitem{ArSa}
		Aroca, F., Saavedra, V.M.:
		Puiseux Parametric Equations via the Amoeba of the Discriminant.
		In: Cisneros-Molina, J. at al. (eds),
		Singularities in Geometry, Topology, Foliations and Dinamics. Trends in Mathematics. Birkhauser, Cham. (2017)	
		
		
		\bibitem{Ky17}
		Kulikov, V.R.:
		A criterion for the convergence of the Mellin-Barnes integral
		for solutions to simultaneous algebraic equations.
		Siberian Math. J. \textbf{58}:3, 493-499 (2017)
		
		
		\bibitem{KuSt}
		Kulikov, V.R., Stepanenko. V.A.:
		On solutions and Waring's formulae for the system
		of n algebraic equations with n unknowns.
		St. Petersburg Math. J. \textbf{26}:5, 839-848 (2015)
		
		
		\bibitem{Me}
		Mellin, Hj.: R\'esolution de l'\'quation alg\'briaue g\'n\'ral \`a l'aide de la fonction $\Gamma$.
	   C.R. Acad. Sci. \textbf{172}, 658-661 (1921)
		
		
		\bibitem{PSTs05}
		Passare, M., Sadykov, T.M., Tsikh, A.K.:
		Nonconfluent hypergeometric functions in several variables and their singularities.
		Compos. Math. \textbf{141}:3, 787-810 (2005)
		
		
		\bibitem{P96}
		Prasolov, V.V.:
		Problems and theorems in linear algebra.
		Translations of Mathematical Monographs, vol. 134, p. 225. AMS (1994)
		
		\bibitem{Pu}
		Puiseux, V.: Rechercher sur les fonction alg\'ebriques.
		J. de math. pures et appl., \textbf{15}, 365-480 (1850) 		
		
		\bibitem{SdTs}
		Sadykov, T.M., Tsikh, A.K.: Hypergeometric and algebraic functions in several
		variables.
		Nauka, M., 408 pp. (2014) (in Russian)
		
		
		\bibitem{St03}
		Stepanenko, V.A.:
		The solution of a system of n algebraic equations in n unknowns by means
		of hypergeometric functions.
		Vestnik Krasnoyar. Gosudarst. Univer. \textbf{1}, 35-48 (2003) (in Russian)
		
		
		\bibitem{ThdeW}
		Theobald, T., de Wolff, T.:
		Norms of Roots of Trinomials. 
		Math. Ann. \textbf{366}, 219-247 (2016)
		
		
		\bibitem{Ts92}
		Tsikh, A.K.:
		Multidimensional residues and their applications.
		Providence RI: Amer. Math. Soc., (1992)	
		

		\bibitem{ZhTs}
		Zhdanov, O.N., Tsikh, A.K.:
		Investigation of multiple Mellin-Barnes integrals by means of multidimensional residues.
		Siberian Math. J. \textbf{39}:2 245-260 (1998)
	
	\end{thebibliography}
\end{document}